\documentclass[a4paper,12pt]{article}
\usepackage{amsmath}
\usepackage{amssymb}
\usepackage{amsthm}
\usepackage{graphicx}
\usepackage{mathrsfs}
\newtheorem{theorem}{Theorem}[section]

\newtheorem{proposition}[theorem]{Proposition}

\newtheorem{remark}[theorem]{Remark}

\newcommand{\PP}{\mathbb{P}}

\newcommand{\RR}{\mathbb{R}}
\newcommand{\ZZ}{\mathbb{Z}}

\newcommand{\E}{{\cal E}}

\newcommand{\G}{{\cal G}}

\newcommand{\wpi}{{\widehat{\pi}}}
\newcommand{\wP}{{\widehat{P}}}

\newcommand{\cO}{{\cal O}}

\newcommand{\U}{{\cal U}}

\newcommand{\mS}{{\mathfrak{S}}}
\newcommand{\vp}{{\vec{p}}} 
\newcommand{\vb}{{\vec{b}}}

\newcommand{\1}{{\bf{1}}}
\newcommand{\me}{\mathfrak{e}}

\newcommand{\mN}{\mathfrak{N}}

\begin{document}

\title{\itshape Maxentropy completion and properties of 
some partially defined Stationary Markov chains}


\author
{Pierre Collet$^{\textup{(a)}}$, 
Servet Mart\'{\i}nez$^{\textup{(b)}}$\\
{\small $^{\textup{(a)}}$ CPHT, CNRS, Ecole Polytechnique, 
IP Paris, Palaiseau, France}\\
{\small e-mail: pierre.collet@polytechnique.edu}\\
{\small $^{\textup{(b)}}$ Universidad de Chile,
 UMR 2071 UCHILE-CNRS}\\
{\small Departamento de Ingenier\'{\i}a Matem\'atica
and Centro de Modelamiento Matem\'atico}\\
{\small Casilla 170-3 Correo 3, Santiago, Chile}\\
{\small e-mail: smartine@dim.uchile.cl}\\
}

\maketitle

\begin{abstract}
  We consider a stationary Markovian evolution with values on a
  finite disjointly
  partitioned set space $I\sqcup \E$.  The evolution is visible (in the
  sense of knowing the transition probabilities) on the states in $I$
  but not for the states in $\E$.  One only knows some partial
  information on the transition probabilities on $\E$, the input and
  output transition probabilities and some constraints of the transition probabilities on
  $\E$. Under some conditions we supply the transition probabilities
  on $\E$ that satisfies the maximum entropy principle.
\end{abstract}

\bigskip




\section{Introduction}
Discrete time Markov chains on a finite state space are completely
described by their matrix of transitions probabilities. We consider a situation
where only a subset of the transition probabilities is known.
More precisely we  consider an evolution on a set of states $I\sqcup \E$
(disjoint union).
The states in  $I$ are visible  and the set $\E$ is called a 
labyrinth (composed of invisible states). 
In a first stage the states in $\E$ act
as absorbing states but their persistence or reappearance 
 lead to consider these states in a different way. 
The transition probabilities  within the labyrinth are unknown. 
One only knows some basic relations: the input and output 
transition probabilities  and eventually  communication constraints in the
labyrinth. The transition probabilities between the states in $I$ are
known and those between $I$ and $\E$ also. 
We assume $I$ and $\E$ are nonempty sets
(there 
are two mentions to the case $I=\emptyset$, in Remarks \ref{rem3.0}) and \ref{rem3}).

\medskip

\noindent A natural question is to complete the description of the chain,
i.e. to find the 
transition probabilities into the labyrinth so that the matrix of
transition probabilities is Markovian. There are many possible
solutions to this problem and 
we propose a  maximum entropy approach.

\medskip

\noindent This partially defined process is an extremely simplified
model for some phenomena:

\smallskip

\begin{itemize} 

\item Disease epidemics that emerge and disappear
in time. A disease may emerge at some location due to human 
contact with a biological reservoir of some microorganism, 
diffuse in some geographical area, and when finished
retreat unseen to some biological reservoir again.
In this case the observable states represent the sizes of 
the healthy and infected populations in geographic areas 
while the labyrinth corresponds to the
biological reservoirs of the microorganism. A transition from 
the labyrinth to an observable state is an outbreak;

\item Underground rivers, that can disappear at some sinkholes, 
continues underground and reemerge further downstream. There can 
be a cave system having an underground water dynamics. At a first 
stage the caves can be  considered absorbing states, but this 
does not describe the whole phenomenon because the river reemerges
after an underground water evolution.

\end{itemize}

\smallskip

\noindent One can conceive many other concrete examples of partially
defined Markov processes.

\medskip

\noindent We assume a stationary behavior of the
total process, not only  on the visible states $I$ 
but  also  for the evolution on the states $\E$
(which is unseen). Given the 
data (some transition probabilities associated with the states in $I$),
we try to reconstruct
the transition probabilities of the process in the unseen part by
using the maximum entropy principle. Our problem falls quite naturally
in this context. Some average quantities (probability of two successive
states) are known and we asked for some other similar quantities. As
described by Jaynes, "...in making inference on the basis of partial
information we must use that probability distribution which has the
maximum entropy to whatever is known.  This is the only unbiased
assignment we can make..." \cite{jaynes} p.623.  See also 
\cite{jaynes1}, and \cite{Sivia} Section 5.

\medskip

\noindent Therefore, we assume we are given some transition probabilities
and the proposal for  the remaining unknown ones is based upon 
the maximum entropy principle. 
We emphasize that we do not assume 
sample paths of the Markov chain are available, 
in other words we do not consider observations of a process  
and no estimation is made.
There is a large literature devoted to estimate a transition matrix that fits 
incomplete observation data, by using the maximum likelihood 
criterion, like in \cite{kd},  or by some Bayesian approach, where
there is a prior of the transition matrix  
and given the data one seeks to compute the posterior distribution,
see \cite{pfb}.  

\medskip

\noindent The paper is organized as follows. In section \ref{model} we describe
the set of known transition probabilities  and introduce some more
hypothesis. We then derive the maximum entropy
solution satisfying the Markovian constraints.  
We are aware that the hypothesis $(H1)$ and $(H2)$ are very strong, but up no now
they turn to be necessary for describing the 
maxentropy behavior inside the labyrinth.

\medskip

\noindent Section \ref{secSt} is devoted to 
the more general case   
where  a communication matrix inside the labyrinth is imposed.
If all the states communicate the maximum entropy solution is Bernoulli. But the
problem turns harder when the communication matrix is not trivial. In fact
there could be no matrix that satisfies the constraints. In Section \ref{Secmc}
 the constraints are written in a linear way and in Section \ref{Secal} we 
use Farkas' lemma to give a criterion for the existence of a solution. In Proposition \ref{propp2} 
we give a simple sufficient condition for the existence of solution. 
In Section \ref{SecNTL} we study the case when there are several labyrinths
that are connected only by the visible states, and show that all the computations related to
maximum entropy matrices are reduced to the case when there is a unique labyrinth. 

\medskip

\noindent  In section \ref{qsd} we relate our problem and hypothesis to the
theory of quasi stationary distributions.  We show that under our
hypotheses the chain 
can be reconstructed 
from the evolution on the visible states and
in the labyrinth  
separately, using families of independent geometric random variables
that mark the times of restarting.





\medskip

\noindent 
As explained above we assume 
 only a subset of the
transition probabilities is known and we look to complete these transitions
probabilities by using the maximal entropy criterion.  This makes part
of the problem where the transition matrix is known to belong to some
class of matrices, and one seeks the one that satisfies a certain
optimization criterion. This is the case in \cite{data}, but the
optimization is on an infinite horizon reward on all the states.  
In \cite{incertain}  the 
unknown is a transition matrix that satisfies some general order
constraints on the coefficients and its stationary vector, which is
unknown, maximizes some linear functional. But, in our case the set of
constraints is more specific allowing to handle the optimization of
the entropy.

\medskip

\noindent In \cite{NIPS} the inverse problem for partially observed Markov
chains with restart has been considered in a parametric setting for the unknown
transition probabilities and with an objective (cost)  function different from
ours.  
More precisely, 
there is a transition probability kernel $p(\cdot,\cdot)$ among the states and a distribution ${\tilde p}$ 
on the states, and at each time the chain evolves with the kernel $p(\cdot,\cdot)$ with probability 
$(1-\beta)$ or it restarts with ${\tilde p}$ with probability $\beta\in (0,1)$.
%
The process is assumed
to be ergodic and the authors look for for a stationary
probability. This search is based upon the knowledge of the stationary
vector up to a constant and the hitting probabilities prior to
restart, in a fixed region. The optimization is made on the parameters
guiding the evolution.  There are some differences with our setting, 
one of them 
is that the optimization functions are different. 

\medskip


\section{Partially defined Markov chains, hypothesis 
and Maxentropy. }
\label{model}

\noindent Our starting point is an irreducible stationary Markov chain, 
$X=(X_n: n\ge \ZZ)$ with values in the finite set $I\sqcup \E$. Its transition
matrix is denoted by $P=(P(a,b): a,b\in I\sqcup \E)$ and $\pi$ denotes  
stationary distribution, so $\pi^t P=\pi^t$. 
The irreducibility of $P$ implies $\pi>0$.
We will assume that the restricted transition matrix $P_{I,I}=(P(a,b): a,b\in I)$ is also an 
irreducible matrix.

\medskip

\noindent By vectors we always mean 
column vectors and we add the transposition superscript ${}^t$  
to denote the associated row vector. We also denote
$P_{J\times K}=(P(a,b): i\in J, k\in J)$ the restricted submatrix
when $J,K\subseteq I\sqcup \E$. 

\medskip

\noindent When the chain $X$ takes values in $I$, this is
visible. The states in the labyrinth $\E$ can be distinguished and
visible only when there
is a nonzero probability of  emerging from $\E$ to $I$ or when there 
is a nonzero probability of transition  from
some state in $I$ to $\E$.
But once the chain is  inside and before leaving  the labyrinth 
this is not visible.

\begin{remark}
\label{rem0}
Let us see that the pair formed by the visible process $\cO$ (defined
below) and the 
Markov process $X$ does not constitute a Hidden Markov chain.  
For  this purpose let $\me$ be a new state, that 
codes the states of the process that are in $\E$ but they are not seen during the process. 
With this notation, the visible process $\cO=(\cO_n:n\in \ZZ)$ is given by
$$
\cO_n=
\begin{cases}
X_n & \hbox{ if } X_n\in I;\\
X_n & \hbox{ if } X_n\in \E \hbox{ and } X_{n-1}\in I \hbox{ or } 
X_{n+1}\in I;\\
\me & \hbox{ if } X_n\in \E \hbox{ and } X_{n-1}\in \E \hbox{ and }
X_{n+1}\in \E.
\end{cases}
$$
A necessary condition in order that the pair $(\cO,X)$ is 
a Hidden Markov chain 
is that: for all $n\ge 1$, $c\in I\sqcup \E\sqcup \{\me\}$, 
$a_k\in I\sqcup \E$, 
$k=0,..,n$, one has:
\begin{equation}
\label{pomp}
\PP(\cO_n=c \, | \, X_k=a_k, k\le n)=\PP(\cO_n=c \, | \, X_n=a_n).
\end{equation}
\noindent But,  this condition 
is not satisfied when $c=\me$, $a_{n-1}\in I$, $a_n\in \E$. In fact,
since $X_{n-1}=a_{n-1}\in I$ then the visible state at coordinate $n$ is 
$\cO_n=X_n=a_{n}\in \E$ which is different from $\me$, and so 
$\PP(\cO_n=\me | X_k=a_k, k\le n)=0$. But 
we can have
$\PP(\cO_n=\me| X_n=a_n)=
\PP(X_{n-1}\in \E, 
X_{n}\in \E, X_{n+1}\in \E | X_n=a_n)>0$. This 
is fulfilled in our case 
when using 
the hypothesis $(H2)$ and the condition $\pi_\E>0$ stated below. 
Also, when it exists, 
the maximum entropy matrix $P_{\E\times \E}$ fulfills this condition, because
the labyrinth network is assumed to be irreducible (See Theorem \ref{maintheorem}).

\medskip

\noindent In \cite{king}, relation $(1)$ in Section 2,  the above 
condition (\ref{pomp}) is part of  
the Partially Observed Markov Process (POMP) model.
For the definition and applications of Hidden Markov models for
discrete random sequences one can see Sections $13.1$ and $13.2$ in 
\cite{bishop}. $\Box$
\end{remark}

\noindent Let $\pi_I=(\pi(i): i\in I)$ 
and $\pi_\E=(\pi(d):d\in \E)$ denote 
the restrictions of $\pi$ to $I$ and $\E$ respectively.

\medskip

\noindent Summarizing, we  assume 
that the transition submatrices $P_{I\times I}$, $P_{I\times \E}$ and
$P_{\E\times I}$ are known. In particular 
$P_{I\times (I\sqcup \E)}$ is known. We also assume that $\pi_I$ is
known. These are the only available data.

\medskip

\noindent If  a large number of observations of the 
dynamics of the chain is available, these data can be estimated 
 by averaging using the Law of Large Numbers
when the chain is  in $I$ and when 
it enters from $I$ to $\E$ or emerges from $\E$ to $I$. 

\medskip

\noindent We denote by $\pi(a)$ the weight of $a\in I\sqcup \E$ and by 
$\pi(A)=\sum_{a\in A} \pi(a)$ the weight of $A\subseteq I\sqcup \E$. 

\medskip

\noindent We assume the existence of some states $i\in I$
such that $P(i,\E)>0$ (where $P(i,\E)=\sum_{d\in \E} P(i,d)$), 
hence the restricted matrix $P_{I\times I}$ 
is a substochastic kernel. 

\medskip

\noindent The first assumption we make is the following one: 
when the chain emerges from the labyrinth $\E$
to the visible  states $I$, it is always with the stationary law 
$\pi_I$, so we assume 
$$
(H1) \quad \quad\quad \quad\quad\quad P_{\E\times I}={\1}_{\E}\pi^t_I
$$
where ${\1}_{\E}$ is the unit vector of dimension the cardinal 
number $|\E|$. This hypothesis means that the law of the outbreak on 
$I$ has distribution $\pi^t_I$. In the case where a sample of the chain
is available 
 one could look for elaborating a test 
that permits to analyze if hypothesis $(H1)$ is satisfied.

\medskip

\noindent Now we give a context for the next assumption.
Since $\pi$ is stationary one has 
$\pi_I^t
=\pi_I^t P_{I\times I}+\pi_{\E}^t P_{\E\times I}$ and 
by using $(H1)$ one gets 
$$ 
\pi_I^t=\pi_I^t P_{I\times I}+\pi_{\E}^t {\1}_{\E} \pi^t_I
=\pi_I^t P_{I\times I}+\pi(\E)\pi^t_I. 
$$ 
(Recall $\pi(\E)=\sum_{d\in \E}\pi(d)$, $\pi(I)=\sum_{i\in I}\pi(i)$). 
Since $1-\pi(\E)=\pi(I)$, then a consequence 
of $(H1)$ is 
\begin{equation}
\label{Eq1'}
\pi_I^t P_{I\times I}=\pi(I) \pi_I^t, 
\end{equation}
and so $\pi_I$ is the left Perron-Frobenius eigenvector 
of $P_{I\times I}$ with Perron-Frobenius eigenvalue $\pi(I)$.

\medskip

\noindent We have only partial information on the transition matrix 
$P_{\E\times \E}$. 
First, one has that all sums of the rows of $P_{\E\times \E}$ are 
constant because $P$ is Markov and so from $(H1)$,
$$
{\1}_{\E}=P_{\E\times I} {\1}_{I}+P_{\E\times \E}{\1}_{\E}=
{\1}_{\E}\pi_I^t {\1}_{I}+P_{\E\times \E}{\1}_{\E}
=\pi(I){\1}_\E+P_{\E\times \E}{\1}_{\E}.
$$
Then,
\begin{equation}
\label{Eq1}
P_{\E\times \E}{\1}_{\E}=\pi(\E){\1}_\E.
\end{equation}
Therefore, a consequence of $(H1)$ is that 
$\pi(\E)^{-1}P_{\E\times \E}$ is a stochastic matrix. 
We have that $\pi(\E)^{-1}\pi_{\E}$ is a probability vector. 
We assume the following property:  
$\pi(\E)^{-1}\pi_{\E}$ is a stationary distribution for 
$\pi(\E)^{-1}P_{\E\times \E}$. This  is equivalent to,
$$
(H2) \quad \quad\quad\quad\quad\quad 
\pi(\E) \pi_\E^t=\pi_\E^t P_{\E\times \E}.
$$
So, $\pi_\E$ is the left Perron-Frobenius eigenvector of 
$P_{\E\times \E}$ with Perron-Frobenius eigenvalue $\pi(\E)$.

\medskip

\noindent Note that the hypothesis $(H1)$ can be tested from observing 
the transitions  from the labyrinth to the visible states, 
but this is not the case of $(H2)$. 

\medskip

\noindent From (\ref{Eq1}) and $(H2)$ we have 
$\pi_\E^t P_{\E\times \E}{\1}_\E=\pi(\E)^2$. 
The stationarity of $\pi$ implies
$$
\pi_\E^t=\pi_I^t P_{I\times \E}+\pi_{\E}^t P_{\E\times \E}
$$
and so from $(H2)$ one gets
$$
\pi_I^t P_{I\times \E}=\pi_{\E}^t-\pi_\E^t P_{\E\times \E}
=(1-\pi(\E))\pi_\E=\pi(I)\pi_\E.
$$
and then
\begin{equation}
\label{Eq3}
\pi_{\E}=\pi(I)^{-1} \pi_I^t P_{I\times \E},
\end{equation}
that is $\pi(e)=\pi(I)^{-1} \sum_{i\in I} \pi(i) P(i,e)$ for $e\in \E$.

\medskip

\noindent Hence, from hypotheses $(H1)$ and $(H2)$ one knows
$\pi_\E$. As said, we have  $\pi_I$ and $P_{I\times \E}$ from
the visible evolution. But, based upon observation of a sample, we cannot 
test that $\pi(\E)^{-1}\pi_\E$ is the stationary distribution for 
$\pi(\E)^{-1} P_{\E\times \E}$, which gives the 
dynamics conditioned to be in $\E$. This occurs even if $(\pi_I,\pi_\E)$ 
is the stationary distribution for the dynamics on $I\sqcup \E$. 
This is why the hypothesis $(H2)$ cannot be tested.

\medskip

\noindent The entropy of the stationary Markov chain $X=(X_n)$
is given by, 
$$
h(X)=-\sum_{a\in I\sqcup \E} \pi(a)
\sum_{b\in I\sqcup \E}P(a,b) \log P(a,b),
$$
so, using (H1)
\begin{eqnarray}
\nonumber
h(X)&=&-\sum_{d\in \E} \pi(d)
\sum_{b\in I\sqcup \E}P(d,b) \log  P(d,b)
-\sum_{i\in I} \pi(i) \sum_{b\in I\sqcup \E} P(i,b) \log  P(i,b)\\
\nonumber
&=& -\sum_{i\in I} \pi(i) \sum_{b\in I\sqcup \E} P(i,b) \log  P(i,b)
-\sum_{d\in \E} \pi(d) \sum_{j\in I}\pi(j)\log \pi(j)\\
\label{Eq4}
&{}&\; -\sum_{d\in \E} \pi(d) \sum_{e\in \E} P(d,e) \log  P(d,e).
\end{eqnarray}

\noindent According to  Jaynes maximal entropy principle 
 we look for maximizing $h(X)$, with 
the following knowledge:
$$
\pi_I, P_{I\times (I\cup \E)}, P_{\E\times I}. 
$$  
So, the first term of the (\ref{Eq4}) is fixed.  
Now,  from assumptions $(H1)$ and $(H2)$ we also know $\pi_\E$.
So, the second term in  (\ref{Eq4}) is also fixed. Hence, 
the maximization of $h(X)$ is equivalent to maximize 
$$
H'(P_{\E\times \E})=-\sum_{d\in \E} \pi(d) \sum_{e\in \E} P(d,e) \log  
P(d,e),
$$
with $\pi_\E=\pi(I)^{-1} \pi_I^t P_{I\times \E}$ known. So, one 
seeks to
maximize $H'(P_{\E\times \E})$ with the matrix $P_{\E\times \E}$
subject to (\ref{Eq1}) and $(H2)$. 

\medskip

\section{Partially defined Markov chains with communication constraints.}
\label{secSt}
This section  will be mainly devoted to the case where there are some communications constraints
among the states of the labyrinth. Firstly we will consider the case when there 
is no constraint of this type, in other words  the states in $\E$ communicate among themselves.

\subsection{When all states in the labyrinth communicate} 

\begin{theorem}
When the unique constraints are (\ref{Eq1}) and $(H2)$, 
the maximum entropy completion of the  chain is given by  
$P_{\E\times \E}={\1}_\E \pi_\E^t$, 
that is the transitions are Bernoulli: $P(d,e)=\pi(e)$, $d,e\in \E$. 
\end{theorem}

\begin{proof}
When the unique constraints are (\ref{Eq1}) and $(H2)$, 
the maximum of $H'(P_{\E\times \E})$ is obtained with 
$P_{\E\times \E}={\1}_\E \pi_\E^t$.
This is well-known and can be proved as follows,
$$
H'(P_{\E\times \E})=\sum_{d\in \E} \pi(d) \sum_{e\in \E} P(d,e) \log  
(\pi(e)/P(d,e))-\sum_{d\in \E} \pi(d)\sum_{e\in \E} P(d,e) \log \pi(e).
$$
By using $\log(x)\le x-1$ and (\ref{Eq1}) we get,
\begin{eqnarray*}
&{}& \sum_{d\in \E} \pi(d) \sum_{e\in \E} P(d,e) \log (\pi(e)/P(d,e))
\le \sum_{d\in \E} \pi(d) \sum_{e\in \E} P(d,e) (\pi(e)/P(d,e)-1)\\
&{}& \; \le \sum_{d\in \E} \pi(d) \sum_{e\in \E} \pi(e)- 
\sum_{d\in \E} \pi(d) \sum_{e\in \E} P(d,e)= \pi(\E)^2-\pi(\E)^2=0.
\end{eqnarray*}
Hence, 
\begin{eqnarray*}
H'(P_{\E\times \E})&\le& -\sum_{d\in \E} \pi(d)\sum_{e\in \E} 
P(d,e) \log(\pi(e)) =-\sum_{e\in \E} \log(\pi(e)) 
\sum_{d\in \E} \pi(d)P(d,e)\\
&=&-\pi(\E)\sum_{e\in \E} \pi(e) \log(\pi(e)),
\end{eqnarray*}
where $(H2)$ was used to establish the last inequality.
This inequality is an equality 
for  the transitions $P(d,e)=\pi(e)$ with $d,e\in \E$.
Hence, in the absence of other constraint the maximum of the entropy 
is realized with a Bernoulli.
\end{proof}

\begin{remark}
\label{rem3.0}
The Bernoulli case was described in \cite{sm}. 
When $I=\emptyset$, one has $\pi(\E)=1$, hence the unique restriction $(H2)$ 
on the matrix $P_{\E\times \E}$ implies it is a stochastic matrix. The
maximum entropy matrix is the uniform one $P(d,e)=1/|\E|$ for all  
$d,e\in \E$.$\Box$
\end{remark}

\subsection{Constraints on the communication 
matrix in the labyrinth.}
\label{Secmc}
\noindent Let us define $\wpi$ on $I\sqcup \E$ the normalized restrictions 
to $I$ and $\E$, 
that is $\wpi=\wpi_I$ on $I$ and $\wpi=\wpi_\E$ on $\E$, 
where
$$
\wpi_I=(\pi(I))^{-1}\pi_I \, \hbox{ and } \, \wpi_\E=(\pi(\E))^{-1}\pi_\E.
$$
These are probability vectors on $I$ and $\E$ respectively. When necessary we will extend these
vectors to $I\cup \E$ by putting 
$\wpi_I(d)=0$ for $d\in \E$ and $\wpi_\E(i)=0$ for $i\in I$. In this case, $\pi=\pi(I)\wpi_I+\pi(\E)\wpi_\E$.

\medskip

\noindent We define
$$
\wP=(\pi(\E))^{-1}P_{\E\times \E}.
$$ 
From (\ref{Eq1}), this is a stochastic matrix of size $|\E|\times |\E|$. 
We denote by 
$Z$ a Markov shift with transition matrix $\wP$ and stationary 
distribution $\wpi_\E$. Note that 
\begin{eqnarray*} 
H'(P_{\E\times \E})&=&-\sum_{d\in \E} \pi(d) \sum_{e\in \E} P(d,e) 
\log P(d,e)\\ 
&=&-\pi(\E)^2 \sum_{d\in \E} \wpi(d) \sum_{e\in \E} \wP(d,e) 
\log(\pi(\E)\wP(d,e))\\ 
&=&-\pi(\E)^2 h(Z)+\pi(\E)^2\log(\pi(\E)). 
\end{eqnarray*} 
Hence, the maximization of $H'(P_{\E\times \E})$ such that 
$\pi(\E)$ is known and  (\ref{Eq1}) and $(H2)$
are satisfied, is equivalent to maximize 
the entropy $h(Z)$ of the Markov shift $Z$, 
subject to the knowledge of its stationary distribution $\wpi_\E$.

\medskip

\noindent We will assume one also knows the communication matrix 
$L=(L(e,d): e,d\in \E)$ in the labyrinth. This information 
is $0-1$ valued: 
$L(e,d)=1$ means the state $d$ communicates  
with $e$ in the labyrinth, and  $L(e,d)=0$ means this communication is
forbidden. We assume the labyrinth  
is an irreducible network, namely  
$L$ is an irreducible matrix. 
In other words
for all $d,e$ there is a path $d_0=d,..,d_k=e$ such that $L(d_i,d_{i+1})=1$
for $i=0,..,k-1$. For $d\in \E$ we define,
$$
L_d=\{e\in \E: L(d,e)=1\} \hbox{ and } L^d=\{e\in \E: L(e,d)=1\},
$$
which are the set of states that follow $d$ and the set of states 
that precede $d$, respectively. By definition, 
$e\in L_d$ if and only if $d\in L^e$. We assume that the transition 
matrix $\wP$ satisfies 
$$
(H3) \quad \quad\quad\quad\quad \quad L(d,e)=0 \Rightarrow \wP(d,e)=0.
$$
So, we must maximize
$$
h(Z)=-\sum_{d\in \E} \wpi(d)\sum_{e\in L_d} \wP(d.e)\log \wP(d,e),
$$
subject to the set constraints
\begin{eqnarray}
\label{Eq9.11}
&{}&\forall d\in \E: \; \sum_{e\in L_d} \wP(d,e)=1 \,;\\
\label{Eq9.21}
&{}& \forall e\in \E: \; \sum_{d\in L^e} \wpi(d) \wP(d,e)=\wpi(e)\, .
\end{eqnarray}

\begin{theorem}
\label{maintheorem}
Assuming that the constraints \eqref{Eq9.11} and \eqref{Eq9.21} are
feasible, the maximum entropy completion of the  chain satisfying
hypothesis $(H3)$ is given by 
$$
\wP(d,e)=\alpha(d)\beta(e) {\bf 1}(e\in L_d)\;,
$$
where 
\begin{equation}
\label{Eq8}
\forall e\in \E: \; \beta(e)=\frac{\wpi(e)}{\sum_{d\in L^e} \wpi(d) 
\frac{1}{\sum_{c\in L_d} \beta(c)}}\;,
\end{equation}
and
\begin{equation}
\label{Eq6}
\forall d\in \E: \; \alpha(d)=\frac{1}{\sum_{e\in L_d} \beta(e)}\;.
\end{equation}
The function $\alpha$ and $\beta$ are strictly positive and the matrix $\wP$ is irreducible.
Moreover, the condition 
\begin{equation}
\label{suf}
\forall d\in \E:\quad L(d,d)=1,
\end{equation}
is sufficient in order that the
constraints \eqref{Eq9.11} and \eqref{Eq9.21} are feasible, that is
under (\ref{suf})  there always exists a nonnegative matrix $\wP$ satisfying them.
\end{theorem}

\begin{proof}
By using Lagrange multipliers one gets that the maximum entropy is 
attained for transition probabilities of the form
\begin{equation}
\label{EqN1}
\wP(d,e)=\alpha(d)\beta(e) {\bf 1}(e\in L_d).
\end{equation}
The functions $\alpha$ and $\beta$ defined on $\E$ are nonnegative. 
Since the matrix $\wP$ is stochastic one gets that 
$\alpha$ is strictly positive. If some $\beta(e)=0$ one will obtain 
$\wP(d,e)=0$ for all $d\in L^e$ and (\ref{Eq9.21}) gives $\wpi(e)=0$,
hence 
also $\pi(e)=0$, which contradicts the irreducibility of $P$.
Since the constraints are satisfied, one gets:
(\ref{Eq6}) holds and 
\begin{equation}
\label{Eq7}
\forall e\in \E: \; \left(\sum_{d\in L^e} \wpi(d) 
\alpha(d))\right) \beta(e)=\wpi(e).
\end{equation}
So, 
by using (\ref{Eq6}) one gets that
$\beta$ satisfies
(\ref{Eq8}).
From (\ref{Eq6}), $\alpha$ is determined by $\beta$.  
So, the functions $\alpha$ and $\beta$ are strictly positive, 
then from the shape of $\wP$ in (\ref{EqN1}), and since $L$ is irreducible, the irreducibility of $\wP$ 
is obtained.

\medskip

\noindent The last part of the theorem will be shown in Proposition \ref{propp2}
in the next Section.
\end{proof}

\medskip

\noindent Notice that one can 
always assume that certain $\beta(e_0)=1$. In fact we can divide 
both sides of (\ref{Eq8}) by $\beta(e_0)$ and 
we obtain analogous relations 
 but with $\beta(e)/\beta(e_0)$ instead of $\beta(e)$.  
So, there are left $|\E|-1$ parameters to be determined.

\medskip

\noindent We can see that in some cases there is no solution and this 
means that the restrictions given by matrix $(L(d,e): d,e\in \E)$ are not 
the good ones or that the assumptions $(H2)$ is not satisfied (this 
is the one that cannot be observed). Let us give an example that does not
have a solution. Take $\E=\{1,2,3\}$ 
and assume the stationary measure $\wpi$ satisfies
$\wpi(1)>\wpi(2)+\wpi(3)$. Then, one can check that for the 
irreducible matrix $L$ given by $L(d,e)=0$ if $e=d$ and 
$L(d,e)=1$ when $e\neq d$, there is no transition matrix $\wP=(\wP(d,e): 
d,e\in \E)$ that satisfies $\wP(d,e)=0$ when $e=d$ and $\wpi_\E$ 
is a stationary 
distribution of $\wP$.

\medskip

\begin{remark}
\label{rem3}
When $I=\emptyset$, we have $\pi(\E)=1$ so $\wP=P$ and  
the unique restriction on the matrix 
$P_{\E\times \E}$ is given by (\ref{Eq9.11}). 
The maximum entropy matrix is given by the Markov chain defined by Parry
distribution \cite{dgs}, let us describe it. 
Let $\varphi$ and $\nu$ be the right
and left Perron-Frobenius eigenvectors associated to $L$ with eigenvalue
$\lambda$ and normalized by $\sum_{d\in \E} \nu(d)\varphi(d)=1$.
The stochastic matrix $\wP(d,e)=\lambda^{-1}\varphi(e)/\varphi(d)$
has stationary distribution $(\nu(d)\varphi(d): d\in \E)$. Its   
entropy is $\log \lambda$, which is the topological entropy.
It is known to maximize the entropy of all the Markov chains 
whose transition matrices $\wP$satisfy $(H3)$. 
(But also of all stationary distributions 
of the topological Markov shift defined by $L$, see \cite{dgs} Chapter 17). 
$\Box$
\end{remark}

\medskip

\subsection{Associated linear problem}
\label{Secal}
\noindent The hypothesis $(H3)$ and constraints (\ref{Eq9.11}), 
(\ref{Eq9.21}), can be put in the 
following form for the matrix $\wP\ge 0$,
\begin{eqnarray}
\label{Eq9.1}
&{}&\forall d\in \E: \; \sum_{e\in \E} \wP(d,e)=1 \,;\\
\label{Eq9.2}
&{}& \forall e\in \E: \; \sum_{d\in \E} \wpi(d) \wP(d,e)=\wpi(e)\,;\\
\label{Eq9.3}
&{}& \forall d,e\in \E: \; L(d,e)=0\Rightarrow \wP(d,e)=0.
\end{eqnarray}
Let us put the whole problem, in particular these constraints, 
in a vector form. 

\medskip

\noindent Let $\ell=|\E|$. We assume 
$\E=\{1,\dots ,\ell\}$ and so $\wP=\{\wP(i,j): i,j=1,\dots,\ell)$.
We define a column vector $\vp=(p(k); k=1,\cdots,\ell^2)$ 
representing the matrix $\wP$. We set
$$
p(i+(j-1)\ell)=\wP(i,j), i,j=1,\dots,\ell,
$$
so we group sequentially the columns of $\wP$, in fact 
$(p(i+(j-1)\ell): i=1,\cdots,\ell)$ is the $j-$the column of
$\wP$. Since $\wP\ge 0$ one has $\vp\ge 0$. We also write
$\wpi_\E=(\wpi(i): i=1,\dots,\ell)$. Consider the Kronecker function
$$
\delta(a,b)=
\begin{cases}
1 & \hbox{ if } a=b \,; \\
0 & \hbox{ if } a\neq b.
\end{cases}	
$$
Define the following three
matrices $A$, $B$, $C$, of dimension $\ell\times \ell^2$. 
For $r,s,t\in \{1,\dots, \ell\}$ set,
\begin{eqnarray*}
&{}& A(r,t+(s-1)\ell)=\delta(t,r); \\
&{}& B(r,t+(s-1)\ell)=\wpi(t)\delta(s,r); \\
&{}& C(r,t+(s-1)\ell)=(1-L(t,r))\delta(s,r).
\end{eqnarray*}
We have 
$$
\sum_{t=1}^\ell \sum_{s=1}^\ell A(r,t+(s-1)\ell)p(t+(s-1)\ell)=
\sum_{s=1}^\ell p(r+(s-1)\ell)=\sum_{s=1}^\ell \wP(r,s).
$$
Then, the restriction $\wP{\1}_\E={\1}_\E$ (\ref{Eq9.1}),  
is equivalent to $A \vp={\1}$ with ${\1}$ the unit vector of size $\ell$.
On the other hand
$$
\sum_{t=1}^\ell \sum_{s=1}^\ell B(r,t+(s-1)\ell)p(t+(s-1)\ell)=
\sum_{t=1}^\ell \wpi(t) p(t+(r-1)\ell)=\sum_{t=1}^\ell 
 \wpi(t) \wP(t,r).
$$
Then,  the restriction $\wpi^t \wP=\wpi^t$ (\ref{Eq9.2}) is equivalent to 
$B \vp=\wpi$. Finally, since $L(\cdot,\cdot)$ is $0-1$ valued
we get,
\begin{eqnarray*}
&{}& \sum_{t=1}^\ell \sum_{s=1}^\ell C(r,t+(s-1)\ell)p(t+(s-1)\ell)
= \sum_{t=1}^\ell (1-L(t,r)) p(t+(r-1)\ell)\\
&{}& =\sum_{t=1}^\ell
(1-L(t,r)) \wP(t,r)=\sum_{t\in \{1,..,\ell\}: L(t,r)=0} \wP(t,r).
\end{eqnarray*}
Since the restriction (\ref{Eq9.3}) 
is equivalent to $L(t,r)=0$ implies $\wP(t,r)=0$ for all $t,r$, then 
it is equivalent to  $C \vp={\bf 0}$, the zero vector of size $\ell$. 

\medskip

\noindent We have proven that there exists 
$\wP\ge 0$ such that the conditions (\ref{Eq9.1}), (\ref{Eq9.2}), 
(\ref{Eq9.3}), are fulfilled if and only if we have,
$$
D\vp=\vb, \; \vp\ge 0,
$$ 
where 
$$
D=(D(r,t+(s-1)\ell): r=1,..,3\ell, t,s=1,..,\ell)
$$ 
is a $3\ell\times \ell^2$ matrix whose coefficients are given by
$$
D(r,t+(s-1)\ell)=
\begin{cases}
A(r,t+(s-1)\ell) & \hbox{ if } 1\le r\le \ell,\\
B(r-\ell,t+(s-1)\ell) & \hbox{ if } \ell+1 \le r \le 2\ell,\\
C(r-2\ell,t+(s-1)\ell) & \hbox{ if } 2\ell+1 \le r \le 3\ell\,;
\end{cases}
$$
and $\vb$ is a $3 \ell-$ dimensional vector given by
$$
b(r)=
\begin{cases}
1 & \hbox{ if } 1\le r\le \ell,\\
\wpi(r-\ell) & \hbox{ if } \ell+1 \le r \le 2\ell,\\
0  & \hbox{ if } 2\ell+1 \le r \le 3\ell\,.
\end{cases}
$$
By Farkas' Lemma (see Proposition 3.2.1 p. 170 in \cite{bno}) 
there exists $\vp\ge 0$ satisfying $D \vp=\vb$ 
or there exists some $y\in \RR^{3\ell}$ satisfying
$D^t y\ge 0$ and $b^t y<0$. Let us set $y^t=(u^t,v^t,w^t)$ with 
$u,v,w$ be $\ell-$dimensional vector. Then we have, 
$$
b^t y=\sum_{r=1}^\ell \big(u(r)+v(r)\wpi(r)\big)
$$
and so the condition $b^t y<0$ is equivalent to
\begin{equation}
\label{rest1}
\sum_{r=1}^\ell (u(r)+ \wpi(r) v(r))<0.
\end{equation}
We have that,
$$
(y^t D) (t+(s-1)\ell)=u(t)+v(s)\wpi(t)+w(s)(1-L(t,s)). 
$$
Then, the condition $D^t y\ge 0$ is equivalent to
\begin{equation}
\label{rest2}
\forall t,s=1,\dots,\ell: \quad u(t)+v(s)\wpi(t)+w(s)(1-L(t,s))\ge 0.
\end{equation}

\medskip

\begin{proposition}
\label{propp1}
The conditions (\ref{rest1}) and (\ref{rest2}) are 
equivalent to (\ref{rest1}) and (\ref{rest3}), with
\begin{equation}
\label{rest3}
\forall t,s \hbox{ such that } L(t,s)=1: \quad u(t)+v(s)\wpi(t)\ge 0.
\end{equation}
\end{proposition}

\begin{proof}
If (\ref{rest3}) is fulfilled then the choice
$$
w(s)\ge max\{|u(t)|+|v(s)|\wpi(t): t=1,\dots,\ell\}, s=1,\dots,\ell,
$$
implies that  (\ref{rest2}) is satisfied and (\ref{rest1}) is not modified.
\end{proof}

\begin{proposition}
\label{propp2}
If $L(i,i)=1$ for all $i=1,\dots,\ell$ then there always exists a 
matrix $\wP\ge 0$ satisfying (\ref{Eq9.1}), (\ref{Eq9.2}) and (\ref{Eq9.3}).
\end{proposition}

\begin{proof}
From the condition $L(i,i)=1$ and (\ref{rest2}) we get 
$u(i)+v(i)\wpi(i)\ge 0$ for all 
$i=1,\dots,\ell$. Then, (\ref{rest1}) cannot be satisfied. Hence, 
there exists $\vp\ge 0$ satisfying $D \vp=\vb$.
\end{proof}

\noindent Hence, the last part of Theorem \ref{maintheorem} is proven.

\subsection{Non-connected labyrinths}
\label{SecNTL}
\noindent When there is a set of disjoint labyrinths that are only 
connected through the states in $I$ the analysis is analogous. 
Let us briefly sketch it. As before, $X$ is an irreducible stationary
Markov chain,  
with values in the set $I\sqcup \E$ with $\E=\bigsqcup_{m\in M}\E_m$
(disjoint union).
 The transitions among the labyrinths 
satisfy $P_{\E_m\times \E_{m'}}=0$ when $m\neq m'$. Since the chain is irreducible, 
$\E_m$ and $\E_{m'}$ are only connected through $I$.
So, for all $m\in M$ there exists some state $i_m\in I$
such that $P(i_m,\E_m)>0$.

\medskip

\noindent Let $\pi_I=(\pi(i): i\in I)$ 
and $\pi_{\E_m}=(\pi(d):d\in \E_m)$ denote 
the restrictions of $\pi$ to $I$ and $\E_m$ respectively, for $m\in M$.
We  assume 
that the transition submatrices $P_{I\times I}$, $P_{I\times \E}$,
$P_{\E\times I}$ and the vector $\pi_I$ are known. 

\medskip

\noindent The assumption $(H1)$ now reads:
$P_{\E_m\times I}={\1}_{\E_m}\pi^t_I$ for all $m\in M$, so
from all the labyrinths the law of the outbreak on 
$I$ has distribution $\pi^t_I$. 

\medskip

\noindent Equation (\ref{Eq1'}), $\pi_I^t P_{I\times I}=\pi(I) \pi^t$ 
is also deduced from $(H1)$.  From $(H1)$ we also get,
$$
{\1}_{\E_m}=P_{\E_m\times I} {\1}_{I}
+\sum_{m'\neq m}P_{\E_{m}\times \E_{m'}}{\1}_{\E_{m'}}
+P_{\E_{m}\times \E_{m}}{\1}_{\E_{m}}
=\pi(I){\1}_{\E_m}+P_{\E_m\times \E_m}{\1}_{\E_m}.
$$
Since $\pi(\E)=1-\pi(I)$,  (\ref{Eq1}) becomes,
$P_{\E_m\times \E_m}{\1}_{\E_m}=\pi(\E) {\1}_{\E_m}$. 
We assume that  
$\pi(\E_m)^{-1}\pi_{\E_m}$ is a stationary distribution for 
$\pi(\E)^{-1}P_{\E\times \E}$. So $(H2)$ becomes
$\pi(\E) \pi_{\E_m}^t=\pi_{\E_m}^t P_{\E_m\times \E_m}$ for all $m\in M$.

\medskip

\noindent The same arguments show that (\ref{Eq3}) can be  written 
$\pi_{\E_m}=\pi(I)^{-1} \pi_I^t P_{I\times \E_m}$ for all $m\in M$.
Hence, all the vectors $\pi_{\E_m}$ are known. 
The entropy of the stationary Markov chain $X$ is,
\begin{eqnarray*}
\nonumber
h(X)&=&-\sum_{i\in I} \pi(i) \sum_{a\in I\sqcup \E} P(i,a) \log  P(i,a)
-\sum_{d\in \E} \pi(d) \sum_{j\in I}\pi(j)\log \pi(j)\\
&{}&\; -\sum_{d\in \E} \pi(d)\sum_{e\in \E} P(d,e) \log  P(d,e).
\end{eqnarray*}
Since all $\pi_{\E_m}=\pi(I)^{-1} \pi_I^t P_{I\times \E_m}$ are known, 
the maximization of $h(X)$ is equivalent to maximize 
$H'(P_{\E\times \E})=-\sum_{d\in \E} \pi(d) \sum_{e\in \E} P(d,e) \log  
P(d,e)$. 
Since $P_{\E_m\times \E_{m'}}=0$ when $m\neq m'$,
$$
H'(P_{\E\times \E})=
-\sum_{m\in M}\left(\sum_{d\in \E_m} \pi(d) \sum_{e\in \E_m} 
P(d,e) \log  P(d,e)\right).
$$ 
So, it is equivalent to
maximize $H'(P_{\E_m\times \E_m})$ with $P_{\E_m\times \E_m}$
subject to $P_{\E_m\times \E_m}{\1}_{\E_m}\!=\!\pi(\E) {\1}_{\E_m}$ and 
$\pi_{\E_m}\!=\!\pi(I)^{-1} \pi_I^t P_{I\times \E_m}$ for $m\in M$. 
Then, the analysis is the same as before, but for each one of the matrices 
$P_{\E_m\times \E_m}$.

\medskip

\section{Restart and quasi-stationarity}
\label{qsd}
\subsection{Quasi-stationarity}
We can use the quasi-stationary theory, see \cite{cmsm}, to 
illustrate how the chain $X$ restarts
its evolution when it exits from $I$ or from $\E$.

\medskip

\noindent When starting from $X_0\in I$, let $\tau_\E=\inf\{n>0: X_n\in \E\}$ be the hitting time of $\E$.
By iterating (\ref{Eq1'}) 
we get $\wpi_I^t P^n_{I\times I}=\pi(I)^n \wpi_I^t$,
and so $\PP_\wpi(\tau_\E>n)= \pi(I)^n$, Then, 
$$
\PP_{\wpi_I}(X_n=j \, | \, \tau_\E>n)=\wpi_I(j), j\in I, n\ge 0,
$$
that is $\wpi_I$ is a quasi-stationary distribution (q.s.d.) 
of the Markov chain $X^{(I)}=(X_n: n<\tau_\E)$ with states 
in $I$ and killed at $\tau_\E$. So, When
starting from $\wpi_I$ the hitting time of $\E$ is geometrically distributed,
$\tau_\E\sim $Geometric$(\pi(\E))$ (firstly shown in  \cite{fmp} for q.s.d.'s).

\medskip 

\noindent Starting with $\wpi_I$ the exit distribution from $I$ is
given by
\begin{eqnarray*}
\forall d\!\in \!\E\!:&{}&
\PP_{\wpi_I}(X_{\tau_\E}\!=\!d)\!=\! 
\sum_{n\ge 1}\PP(X_{n}=d , \tau_\E\!=\!n)
\!=\!\sum_{n\ge 1}\PP_{\wpi_I}(X_n=d, \tau_\E\!>\!n\!-\!1)\\
&=&\sum_{n\ge 1}  \PP(X_n=d | \tau_\E\!>\!n-1)\PP_{\wpi_I}(\tau_\E\!>\!n\!-\!1)
\!=\!\sum_{n\ge 1}  (\sum_{i\in I}\wpi(i) P(i,d))\pi(I)^{n-1}.
\end{eqnarray*}
From (\ref{Eq3}) we get $\pi(d)\!=\!\sum_{i\in I}\wpi(i) P(i,d)$, and so  
$\PP_{\wpi_I}(X_{\E}\!=\!d)\!=\!\pi(d)(1\!-\!\pi(I))^{-1}$. Then, the exit law from $I$ is $\wpi_\E$:
\begin{equation}
\label{Eqex1}
\forall d\in \E: \quad \PP_{\wpi_I}(X_{\tau_\E}=d)=\wpi(d).
\end{equation}
Similar computations give that when starting from $\wpi_I$, $X_{\tau_\E}$ and $\tau_\E$ are independent, 
$$
\forall d\in \E, n\ge 1: \quad \PP_{\wpi_I}(X_{\tau_\E}=d,\tau_\E=n)=\wpi(d)\PP_{\wpi_I}(\tau_\E=n).
$$

\noindent Similarly, let
$\tau_I=\inf\{n>0: X_n\in I\}$ be the hitting time of $I$. Then, the hypothesis $(H2)$ 
implies that $\wpi_\E$ is a q.s.d. of the Markov
chain $X^{(\E)}=(X_n: n<\tau_I)$ with states in $\E$ killed at 
$\tau_I$. So, $\PP_{\wpi_\E}(X_n=d \, | \, \tau_I>n)=\wpi_\E(d), d\in \E, n\ge 0$,
and $\tau_I\sim $Geometric$(\pi(I))$.
Let us see that when starting with $\wpi_\E$, the exit distribution from $\E$ is $\wpi_I$. For every $i\in I$
we have,
$$
\PP_{\wpi_\E}(X_{\tau_I}\!=\! i)\!=\!\!
\sum_{n\ge 1}\PP_{\wpi_\E}(X_{n}=i | \tau_I\!>\!n\!-\!1)\PP_{\wpi_\E}(\tau_I\!>\!n\!-\!1)
\!\!=\!\sum_{n\ge 1}\! \pi(\E)^{n-1} (\sum_{e\in \E}\!\wpi(e) P(e,i)).
$$
From $(H1)$, $P(e,i)=\pi(i)$ and so $\PP_{\wpi_\E}(X_{\tau_I}\!=\! i)\!=\! \pi(i)(1-\pi(\E))^{-1}$ so, 
\begin{equation}
\label{Eqex2}
\forall i\in I: \quad \PP_{\wpi_\E}(X_{\tau_I}=i)=\pi(I)^{-1}\pi(i)=\wpi(i).
\end{equation}
Also $X_{\tau_I}$ and $\tau_I$ are independent  when starting from $\wpi_\E$,
$$
\forall i\in I, n\ge 1: \quad \PP_{\wpi_\E}(X_{\tau_I}=i,\tau_I=n)=\wpi(i)\PP_{\wpi_\E}(\tau_I=n).
$$
Since $\PP_d(X_1\in \E)=\pi(\E)$ for all $d\in \E$ (see (\ref{Eq1})),
$(X_0,..,X_k, \tau_I>k)$ is independent of $\tau_I>n$ for all $n>k$  
because $\PP(\tau_I>n | X_0,..,X_k)=\pi(\E)^{n-k}$. 

\medskip


\subsection{A construction of the chain}
We will reconstruct a copy of the stationary chain $X=(X_n: n\in \ZZ)$. 
starting from the transitions among the observable states.

\medskip


\noindent  Let $\mS^X=\{n\in \ZZ: X_n\in I\}$. We order these elements, the first nonnegative one is $S^X_0$, so 
$\mS^*=\{S^X_n: n\in \ZZ\}$.  Since $X$ is stationary 
we have, 
\begin{eqnarray}
\label{eqnxid0}
&{}& \forall n\in  \ZZ\setminus \{0\}: S^X_{n+1}-S^X_n\sim  \hbox{Geometric} (\pi(I));\\
\label{eqnxid1}
&{}& S^X_0=\inf\{n\ge 0: X_n\in I\}\sim \hbox{Geometric} (\pi(I))-1;\\
\label{eqnxid2}
&{}&-S^X_{-1}=-\inf\{n< 0: X_n\in I\}\sim  \hbox{Geometric} (\pi(I)). 
\end{eqnarray} 
That is, $\PP(S^X_0=k)=\pi(I)\pi(\E)^n$ for $k\ge 0$. 
The conditions (\ref{eqnxid0}),  (\ref{eqnxid1}), and (\ref{eqnxid2}), are those of 
a stationary renewal sequence in $\ZZ$. They guarantee that,
\begin{equation}
\label{eqnxid3}
\forall t\in \ZZ: \PP(t\in \mS^X)=\pi(I).
\end{equation}
Notice that $1/\pi(I)$ is the mean of a Geometric$(\pi(I))$.
From (\ref{Eq1'}) one gets $X_{S^X_n}\sim \wpi_I$.

\medskip

\noindent For the construction, let us consider the matrix $Q=(Q(i,j): i,j\in I)$ given by
$Q(i,j)=P(i,j)+P(i,\E)\wpi(j)$, namely
$$
Q=P_{I\times I}+P_{I\times \E}{\1}_{\E}\wpi_I^t.
$$
This matrix  is stochastic and one can check that $\wpi_I$ is its stationary distribution.  
Let $Y=(Y_n: n\in \ZZ)$ be a Markov chain with transition matrix $Q$ and stationary 
distribution $\wpi_I$. Then $Y$ gives the trajectories of $(X_{S^X_n}: n\in \ZZ)$.  
In fact, a transition from $i$ to $j$ can be made directly 
with a jump in $P$, plus an entrance to the labyrinth $\E$ and then reemerging from $\E$ to $I$. This
uses (\ref{Eqex1}) and (\ref{Eqex2}) (the matrix $Q$ was introduced in \cite{fkmp}). 

\medskip

\noindent Let us now consider the transitions from $I$ to $\E$, that is the first transition to the labyrinth. 
Let 
$\U=(U_n(i): i\in I, n\in \ZZ)$ be an array of independent random variables taking values 
in $\E$ with 
\begin{equation}
\label{eqc0}
\forall d\in \E, n\in \ZZ: \PP(U_n(i)=d)=P(i,d)/P(i,\E). 
\end{equation}
In particular $\PP(U_n(i)\in \E)=1$. 

\medskip

\noindent Let $Z=(Z_k: 0\le k<\tau)$ be a Markov chain starting from $\wpi_\E$, evolving with $P$ and 
killed when attaining $I$, so $\tau=\tau_I\sim$Geometric$(\pi(I))$ and $Z\sim X^{(\E)}$. Consider 
$(Z^n: n\in \ZZ)$ be a sequence of i.i.d. copies of $Z$ and denote by $\tau^n$ their killing times. So,
$\forall n\in \ZZ: \; \tau^n\sim \hbox{ Geometric}(\pi(I))$.

\medskip

\noindent Define an array $(Z^n(d): d\in \E, n\in \ZZ)$ of independent chains, such that 
$Z^n(d)$ is the chain $Z^n$ conditioned to $Z^n_0=d$. Hence $Z^n_0(d)=d$, it evolves with kernel $P$ 
and its killing time is $\tau^n(d)=\tau^n$. 

\medskip

\noindent To construct a copy 
$W=(W_t: t\in \ZZ)$ of $X$ we cut the trajectory 
$Y$ at some random places 
and insert copies of the killed trajectories on $\E$. This will define a process $W$ 
with states in $I\cup \E$ with the same conditional probabilities as $X$. 
The initial distribution of $W$ will satisfy $W_0\sim \pi$,  
then the processes $W$ and $X$ will be equally distributed.


\medskip

\noindent Let $\G=(G_n(i,j): i\in I, j\in I, n\ge 0)$ be an array of independent Bernoulli random variables with
$G_n(i,j)\sim G(i,j)$ for all $n\ge 0$ and such that $\PP(G(i,j)=1)=\theta(i,j)$, $\PP(G(i,j)=0)=1-\theta(i,j)$,
with
\begin{equation}
\label{eqtheta}
\theta(i,j)\!=\!\frac{P(i,j)}{Q(i,j)}, \;\; (1-\theta(i,j))=\frac{P(i,\E)\wpi(j)}{Q(i,j)}
\end{equation}
Then, $Q(i,j)\theta(i,j)=P(i,j)$ and
$$
\sum_{j\in I}Q(i,j)(1-\theta(i,j))\frac{P(i,d)}{P(i,\E)}=\sum_{j\in I}\wpi(j)P(i,d)=\pi(d),
$$
where the last equality follows from (\ref{Eq3}).
The array $\G$ and a part of the construction that follows was already considered in \cite{sm}.

\medskip

\noindent We define recursively a  process $(W_t: t\in \ZZ)$ 
and a sequence of random times $(S_n: n\in \ZZ)$ which will be  the times of presence of the process 
$W$
in $I$ and the random variables  $\G$ will help us to describe the switching of $W$ from $I$ to $\E$.

\medskip

\noindent Firstly, let $Z'=(Z'_k: 0\le k<\tau')$ be a copy of $Z$, independent of $(Z^n: n\ge 0)$ with
killing time $\tau'\sim \tau_I$.  We define,
$$
W_t=Z'_{t+1} \hbox{ for }0\le t<\tau'-1, \;W_{\tau'-1}\sim \wpi_I \hbox{ and }S_0=\tau'-1.
$$
Therefore $S_0\sim$ Geometric$(\pi(I))-1$
and we have 
$S_0=0$ if and only if $\tau'=1$. So, $\PP(S_0=0)=\PP_{\wpi_\E}(\tau_I=1)=\pi(I)$. 
On the other hand we have, when $S_0=0$ then $W_0\sim \pi_I$ and when $S_0>0$ 
we have $W_0=Z'_1\sim \wpi_\E$.
We conclude that
$$
W_0\sim \pi(I)\wpi_i+\pi(\E)\wpi_\E=\pi.
$$
which is the initial distribution of $X$. The variable $Z'_0\sim \wpi_\E$ is fixed later.

\medskip

\noindent $ \bullet$ Let $n\ge 0$ and assume 
we have defined $(S_l: 0\le l\le n)$, then we continue with step $n+1$. 

\smallskip

\noindent If $G_n(Y_n,Y_{n+1})=1$ then we put $S_{n+1}=S_n+1$, $W_{S_{n+1}}=Y_{n+1}$ and we 
continue with step $n+2$.  

\smallskip

\noindent If $G_n(Y_n,Y_{n+1})=0$ then we put $S_{n+1}=S_n+\tau^n$,  
$W_{S_n+t}=Z^n_{t-1}(U_n(Y_n))$ for $1\le t\le \tau^n-1$ and 
$W_{S_{n+1}}=Y_{n+1}$. In particular 
$W_{S_n+1}=Z^n_0(U_n(Y_n))=U_n(Y_n)$. We continue with step $n+2$.

\smallskip

\noindent This construction for $n=0$ is visualized in Figures $1$ and $2$. 

\medskip

\begin{figure}
\centering
\includegraphics[width=\textwidth]{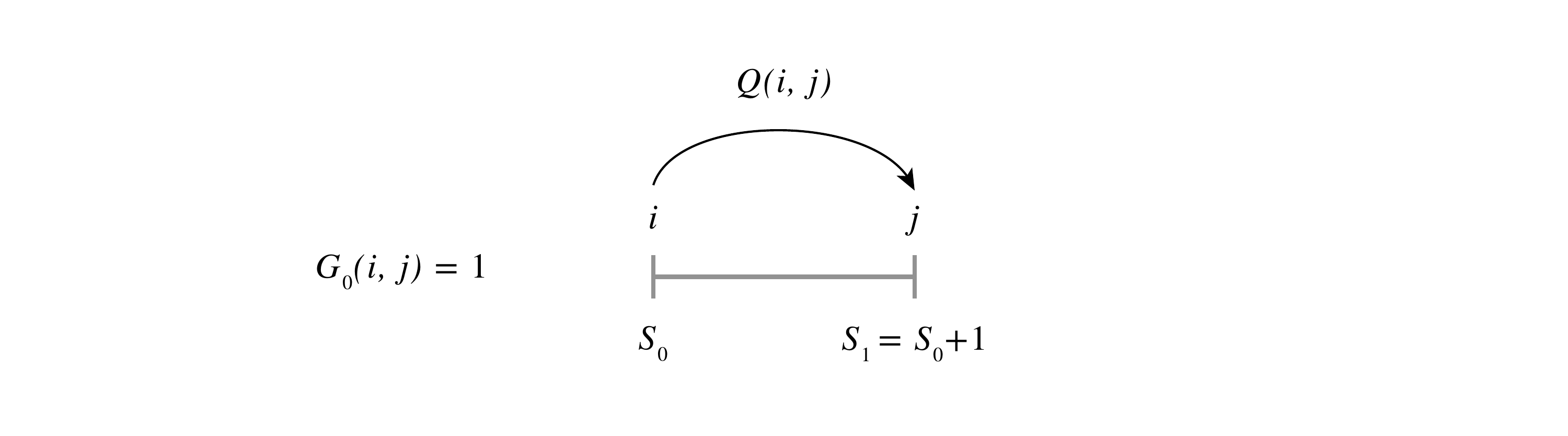}
\caption{$G_{0}(i,j)=1$}
\label{fig5_03}
\end{figure}

\begin{figure}
\centering
\includegraphics[width=\textwidth]{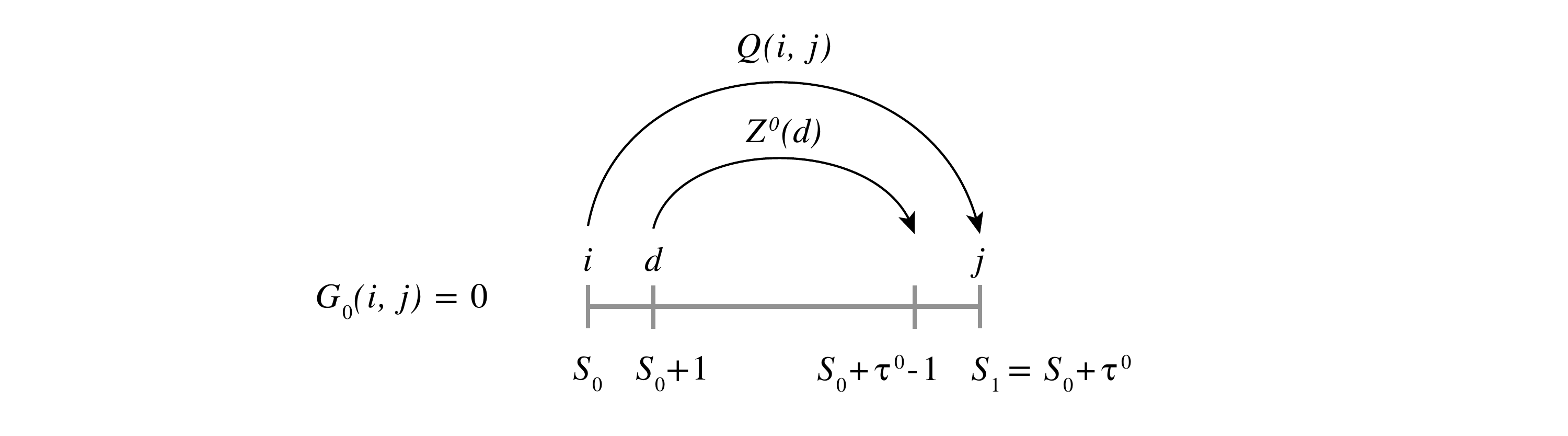}
\caption{$G_{0}(i,j)=0$}
\label{fig6_04}
\end{figure}

\noindent We define $S_{-1}=-\tau^{-1}$ and $W_{S_{-1}}=Y_{-1}$. 
If $S_{-1}=-1$ we take $Z'_0\sim \wpi_\E$
and $Z'$ is chosen independent of the sequence $(Z^n: n\in \ZZ)$.
If $-S_{-1}
\ge 2$,  we take $W_{S_{-1}+t}=Z^{-1}_{t-1}(U_{-1}(Y_{-1}))$ for $1\le t\le \tau^{-1}-1$.
In this case,
in the definition of $S_0$ 
we set $Z'_0=
Z^{-1}_{\tau^{-1}-2}(U_{-1}(Y_{-1}))$.

\medskip

\noindent $\bullet$ Let $n\le -1$ and assume we have already defined 
$(S_l: n\le l\le -1)$, then we continue with step $n-1$. 

\smallskip

\noindent If $G_{n-1}(Y_{n-1},Y_{n})=1$ then we put $S_{n-1}=S_n-1$, $W_{S_{n-1}}=Y_{n}$ and we 
continue with step $n-2$.  

\smallskip

\noindent If $G_{n-1}(Y_{n-1},Y_n)=0$ then we put $S_{n-1}=S_n-\tau^{n-1}$,  
$W_{S_{n-1}+t}=Z^n_{t-1}(U_{n-1}(Y_{n-1}))$ for $1\le t\le \tau^{n-1}-1$ and
$W_{S_{n-1}}=Y_{n-1}$. In particular
$W_{S_{n-1}+1}
=U_{n-1}(Y_{n-1})$. We continue with step $n-2$.

\medskip

\noindent  Hence, we have defined $(W_t: t\in  \ZZ)$ and $\mS=\{S_n: n\in \ZZ\}$. Note that
\begin{equation}
\label{eqinfty} 
\PP(\lim_{n\to \infty} S_n=\infty, \;  \lim_{n\to -\infty} S_n=-\infty)=1.
\end{equation}

\medskip

\begin{theorem}
\label{thmreg}
The  processes $W$ and $X$ have the same law.
\end{theorem}

\begin{proof}
From definition and since $W_{S_0}\sim \wpi_I$, we have $(W_{S_n}: n\in \ZZ)=(Y_n: n\in \ZZ)$. Then, 
\begin{equation}
\label{eqsecond}
\forall 
n\in \ZZ
:  W_{S_n}\sim \wpi_I.
\end{equation}
Let us show that 
\begin{equation}
\forall n\in \ZZ\setminus \{0\}: S_{n+1}-S_n\sim   \hbox{Geometric} (\pi(I)).
\end{equation}  
From the construction, (\ref{eqsecond}) and (\ref{Eq1'}), we have,
$$
\PP(S_{n+1}\!-\!S_n=1)\!=\!\sum_{i\in I}\sum_{j\in I}\wpi(i)Q(i,j)\theta(i,j)\!=\!\sum_{j\in I}\sum_{i\in I}\wpi(j)P(i,j)\!=\!\sum_{i\in I}\pi(j)=\pi(I).
$$
From the construction we have for all $n\ge 2$,
\begin{eqnarray*}
\PP(S_{n+1}-S_n=n)&=&\sum_{i\in I}\sum_{d\in \E}\sum_{j\in I}\wpi(i)Q(i,j)(1-\theta(i,j))\frac{P(i,d)}{P(i,\E)}\pi(\E)^{n-2}\pi(I)\\
&=& \sum_{j\in I}\sum_{i\in I}\sum_{d\in \E}\wpi(i)\wpi(j)P(i,d)\pi(\E)^{n-2}\pi(I)\\
&=&(\sum_{j\in I}\wpi(j))(\sum_{d\in \E}\sum_{i\in I}\wpi(i)P(i,d))
\pi(\E)^{n-2}\pi(I)\\
&=&\pi(I)\pi(\E)^{n-1},
\end{eqnarray*}
where we used $\sum_{j\in I}\wpi(j)P(j,d)=\pi(d)$ that follows from (\ref{Eq3}).
Since $S_0\sim \hbox{Geometric}(\pi(I))-1$ and $-S_{-1}\sim \hbox{Geometric}(\pi(I))$, the sequence 
$(S_n: n\in \ZZ)$ is a stationary renewal sequence in $\ZZ$,
see (\ref{eqnxid0}),  (\ref{eqnxid1}), (\ref{eqnxid2}) and  (\ref{eqnxid3}).
Then, 
\begin{equation}
\label{eqnu1}
\forall t\ge 0: \PP(t\in \mS)=\pi(I).
\end{equation}

\noindent Now define $\mN=\{n\in \ZZ: S_{n+1}-S_{n}\ge 2\}$ that index  
the nonempty connected components of $\ZZ\setminus \mS$. When $n\neq 0$, we have $n\in \mN$ if
$\tau^n\ge 2$. From the construction 
the class of sequences
$\left((W_t: t\!=\!S_n\!+\!1,...,S_{n+1}\!-\!1): n\!\in\!  \mN\right)$ are i.i.d. and
\begin{equation}
\label{eqa3}
\forall n\!\in \!\mN, n\neq 0: 
(W_t: t\!=\!S_n\!+\!1,...,S_{n+1}\!-\!1)\!\sim \!(Z_k: k\!=\!0,...,\tau-2).
\end{equation}
On the other hand if $S_0-S_{-1}>1$ we have,
\begin{equation}
\label{eqa3'}
(W_t: t=S_{-1}+1,...,
{S_0}-1)=(Z^{-1}_0,...Z^{-1}_{\tau^{-1}-2},Z'_1,...,Z'_{\tau'-1}),
\end{equation}
where $Z^{-1}\sim Z(
U_{-1}(Y_{-1}))$ and $Z'=(Z'_0=Z^{-1}_{\tau^{-1}-2},...,Z'_{\tau'-1})\sim Z$.

\medskip

\noindent It is left to show that $W$ is a Markov chain with the same transition 
probabilities as $X$. Let 
$m\in \ZZ$. Assume $W_m=a, W_{m+1}=b$. 

\medskip

\noindent Let $a,b\in I$, or $a,b\in \E$, or $a\in I,b\in \E$. From
the definition of $W$, for all $a_{-l}\in I\cup \E$, $l\ge 1$, and 
for any $m\in\ZZ$ one has
$$
\PP(W_{m+1}\!=\!b|W_m\!=\!a,W_{m-l}\!=\!a_{-l},l\ge 1)\!=\!\PP(W_{m+1}\!=\!b|W_m\!=\!a).
$$
Also for each one of these couples $a,b$ we have  $\PP(W_{m+1}=b|W_m=a)=P(a,b)$. In fact, when $a,b\in I$ 
this follows from $\PP(W_{m+1}=b|W_m=a)=Q(a,b)\theta(a,b)$. If
$a,b\in \E$ this follows from (\ref{eqa3}) and (\ref{eqa3'}). Let us show it for $a\in I,b\in \E$. 
From the construction,
(\ref{eqtheta}) and (\ref{eqc0}), we get,
\begin{equation}
\label{eqc1}
\PP(W_{m+1}=b|W_m=a)\!=\!\sum_{j\in I}Q(a,j) (1-\theta(a,j))\frac{P(a,b)}{P(a,\E)}=\sum_{j\in I}\wpi(j)P(a,b)=P(a,b).
\end{equation}

\medskip

\noindent Let us now take $a\in \E$, $b\in I$. We set $a_0=a$. Notice that
$$
\PP(W_{m+1}\!=\!b |W_{m-l}\!=\!a_{-l}, l\!\ge\! 0)\!=\!\PP(W_{m+1}\!=\!b|W_{m-l}\!=\!a_{-l}, l\!=\!0,..,k),
$$
where $k$ is the first $k\ge 1$ such that $a_{-k}\in I$. 
(From (\ref{eqinfty}) we can assume $k$ to be finite). 
From (\ref{eqnu1}) and (\ref{eqsecond}) we have $\PP(W_{m-k}=a_{-k})=\pi(a_{-k})$, 
Now, we use (\ref{eqtheta}) and (\ref{eqa3}) 
to get
\begin{eqnarray*}
&{}&\PP(W_{m+1}\!=\!b,W_{m-l}\!=\!a_{-l}, l\!=\!0,..,k)\\
&=&\!\!\pi(a_{-k})Q(a_{-k},b)(1-\theta(a_{-k},b))\frac{P(a_{-k},a_{-k+1})}{P(a_{-k},\E)}
\left(\!\prod_{l=1}^{k-1}\!\!P(a_{-k+l},a_{-k+l+1}\right)\pi(I)\\
\!&=&\!\!\pi(a_{-k})\wpi(b)\pi(I)\!\prod_{l=0}^{k-1}P(a_{-k+l},a_{-k+l+1})
=\pi(a_{-k})\pi(b)\!\prod_{l=0}^{k-1}P(a_{-k+l},a_{-k+l+1}).
\end{eqnarray*}
By summing on $b\in I$ we find $\PP(W_{m+1}\in I,W_{m-l}\!=\!a_{-l}, l\!=\!0,..,k)
=\!\pi(I)\pi(a_{-k})\prod_{l=0}^{k-1}P(a_{-k+l},a_{-k+l+1})$. On the other hand when $W_{m+1}\in \E$ we 
have
$$
\PP(W_{m+1}\!\in \!\E, W_{m-l}\!=\!a_{-l}, l\!=\!0,..,k)
=\!\pi(\E)\pi(a_{-k})\prod_{l=0}^{k-1}P(a_{-k+l},a_{-k+l+1}).
$$
Hence
$\PP( W_{m-l}\!=\!a_{-l}, l\!=\!0,..,k)=\!\pi(a_{-k})\prod_{l=0}^{k-1}P(a_{-k+l},a_{-k+l+1})$. So, we get
$$
\PP(W_{m+1}\!=\!b| W_{m-l}\!=\!a_{-l}, l\!=\!0,..,k)=\pi(b).
$$
Since $(H1)$ gives $P(a,b)=\pi(b)$, we have proven 
$\PP(W_{m+1}\!=\!b|W_m\!=\!a,W_{m-l}\!=\!a_{-l}, l=1,..,k)=P(a,b)$.
Hence,
$W$ is a Markov chain having the same transition probabilities as $X$. Then, the proof of Theorem \ref{thmreg} is 
complete.
\end{proof}

\bigskip

\section*{Acknowledgments.}

This work was supported by the Center for Mathematical Modeling ANID Basal
Project FB210005. The authors thank Prof. Jean-Ren\'e 
Chazottes from CNRS for his interest in their work. We also thank an
anonymous referee for pointing out an error in a previous version of
this paper.

\section*{Data Availability Statement.}
No datasets were generated or analysed during the current study.

\section*{Conflict of interest Statement.}
The authors have no competing interests to declare that are relevant
to the content of this article.

\end{document}